\documentclass[12pt,A4,oneside,english]{amsart}
\usepackage{eh}
\usepackage{amsthm}
\newtheorem*{theorem*}{Theorem}
\newtheorem*{lemma*}{Lemma}
\newtheorem*{definition*}{Definition}
\DeclareMathOperator{\Sh}{Sh}
\makeatother
\numberwithin{equation}{section}

\begin{document}
\begin{abstract}
We introduce a geometric method to study additive combinatorial problems. Using equivariant cohomology we reprove the Dias da Silva--Hamidoune theorem. We improve a result of Sun on the linear extension of the Erd\H os--Heilbronn conjecture. We generalize a theorem of G. K\'os (the Grashopper problem) which in some sense is a simultaneous generalization of the Erd\H os--Heilbronn conjecture. We also prove a signed version of the Erd\H os--Heilbronn conjecture and the Grashopper problem. Most identities used are based on calculating the projective degree of an algebraic variety in two different ways.
\end{abstract}
\title{Additive combinatorics using equivariant cohomology}
\author{L\'aszl\'o M. Feh\'er}
\address{Department of Analysis, E\"otv\"os University, Budapest, Hungary}
\email{lfeher@renyi.mta.hu}
\author{ J\'anos Nagy }
\address{Central European University, Budapest, Hungary}
\email{janomo4@gmail.com}
\date{\today}
\thanks{The authors were supported by the grant NKFIH KKP 126683. The second author is also supported by
MTA-BME Lend\"ulet Arithmetic Combinatorics Research Group,
  ELKH, M\H{u}egyetem rkp. 3., H-1111 Budapest, Hungary;
  Department of Computer Science and Information Theory, Budapest University of Technology and Economics, M\H{u}egyetem rkp. 3., H-1111 Budapest, Hungary.
}
\maketitle
\tableofcontents
\section{Introduction}
In the present paper we would like to show a connection between  additive combinatorics and equivariant cohomology and that this connection leads to new results.
\subsection{Known results using equivariant cohomology}
A prototypical theorem in additive  combinatorics is the Dias da Silva--Hamidoune theorem in \cite{Silva-Hamidoune}, conjectured by Erd\H os and Heilbronn:
\begin{theorem*}
Let $A \subset \mathbb{F}_{p}$ be a set, such that $|A| = n$ and $1 \leq k \leq n$, and let us use the notation
  \[    \bigwedge^kA := \left\{ \sum_{i=1}^{ k}a_i\  | \ a_i \in A ,\ a_i \neq a_j  \text{ if } i \neq j \right\}.  \]
Then  $ |\bigwedge^kA| \geq \min\{(n-k)k+1, p \}$ holds.
\end{theorem*}

Most proofs of the Dias da Silva--Hamidoune theorem and similar results in additive  combinatorics uses the so called polynomial method of \cite{anr}, which uses polynomial interpolation formulas, like the Coefficient Formula of \cite{schauz}, \cite{Karasev2012} and \cite{lason} (See also \cite[Lemma 1.1]{karolyi-nagy-petrov}):

\begin{lemma*} Let $f\in \F[x_1,\dots,x_n]$ be a polynomial of degree $\deg(f)\leq d_1+d_2+\cdots+d_n$. For arbitrary subsets $C_1,\dots,C_n$
of $\F$ with $|C_i| = d_i + 1$, the coefficient of $\prod x_i^{d_i}$ in $f$ is
\[\sum_{c_ 1\in C_1 }\sum_{c_ 2\in C_ 2} \cdots \sum_{c_n \in C_n } \frac{f(c_1,c_2,\dots,c_n)}
{\phi^{'}_1 (c_1 )\phi^{'}_2 (c_2 )\cdots \phi^{'}_n (c_n )} ,\]
where $\phi_i(z)=\prod_{c\in C_i }(z-c)$.
\end{lemma*}
 We will review the polynomial method in Section  \ref{sec:coeff} and show that the Coefficient Formula can be interpreted quite naturally as calculating the integral of an equivariant cohomology class over a product of projective spaces in two different ways. This way we can translate the existing polynomial methods to equivariant cohomology.

 But we can do better: Instead of  a product of projective spaces  we can choose an other compact complex homogeneous space which fits better to the additive  combinatorics problem we want to solve. For example the Dias da Silva--Hamidoune theorem leads naturally to an integral over the Grassmannian $\Gr_k(\C^n)$. In this geometric context the number $(n-k)k$ appearing in the estimate of the Dias da Silva--Hamidoune theorem can be interpreted as the dimension of the Grassmannian $\Gr_k(\C^n)$, as it can be seen from \eqref{eq:subst-grass-int}.   We present a new short proof of the Dias da Silva--Hamidoune theorem based on these ideas in  Section \ref{sec:eh}.

\subsection{New results}
For other problems we can explore the geometry of other homogeneous spaces. For example using partial flag manifolds we prove two theorems overlapping with a conjecture of Sun on a linear extension of the Erd\H os--Heilbronn conjecture:

\begin{theorem*}{\bf\ref{sunstrong}}
Let $A$ be a  subset of cardinality $n$ of a field $\F$ with characteristic $ p(\F)$ and let $u_1, \dots, u_k \in \F \setminus 0$. If $p(\F) > d $ then the following holds:
\begin{equation*}
\left|\left\{\sum_{i=1}^k u_i a_i  : a_1, \dots, a_k \in A , \text{ and }a_i \neq a_j \text{ if }i\neq j\right\}\right| \geq d+1,
\end{equation*}
for
\[d:= k(n-k) + \sum_{1 \leq i<j \leq t}k_i k_j,\]
where  $\{k_1,\dots,k_t\}$ denote the multiplicities of the $u_i$'s.
\end{theorem*}
\medskip

\begin{theorem*}{\bf\ref{sun2}}
Let $u_1, \dots, u_ n \in \F $ be different numbers of a field $\F$ with characteristic $ p$ and let $A=\{a_1,\dots,a_n\}$ be a  subset of $\F$.
If $n>3$ and $p  \leq \binom{n}2 $ then the following holds:
\begin{equation*}
\left|\left\{\sum_{i=1}^{n}  u_i a_{\pi(i)}  : \pi \in S_n  \right\}\right| = p.
\end{equation*}
\end{theorem*}
\medskip

Studying the full flag manifold leads us to a  generalization of the Erd\H os--Heilbronn problem in characteristics 0:

\begin{definition*}{\bf\ref{de:admissible}}
A sequence of non negative integers $\mathbf{b}=(b_1,\dots,b_{k-1})$ is admissible if for all $M_i\subset\Z$ and $|M_i|=b_i$, and for all $a_1,a_2,\dots,a_{k}$  distinct integers there is a permutation $\pi\in S_k$ such that for all $j=1,\dots k-1$
\[ \sum_{i=1}^{j}a_{\pi(i)}\notin M_j.\]
\end{definition*}

\begin{theorem*} {\bf\ref{szocsalt-yes}} A sequence $\mathbf{b}$ is admissible if and only if it satisfies the following system of linear inequalities:
for any subset $P \subset \{1, \dots ,k-1 \}$ the  condition
\begin{equation*}\label{hall-system}
   \sum\limits_{p \in P}b_p\ \ \leq \ \ \Big|\{(i,j):\exists\, p\in P, \text{ such that } 1\leq i \leq p \leq j\leq k-1\}\Big|
\end{equation*}
holds.
\end{theorem*}

 We refer to this theorem as the Grasshopper problem for historical reasons explained later. Notice that for sequences $\mathbf{b}$ with only one non zero $b_i$ term, Theorem \ref{szocsalt-yes} specializes to the Erd\H os--Heilbronn theorem in 0 characteristics.

 In Theorem \ref{szocsalt2} we  generalize Theorem \ref{szocsalt-yes} by showing that for small admissible sequences we need only  permutations small in the Bruhat order to avoid the sets $M_i$ by studying the degrees of Schubert varieties in the flag manifold.

All these examples are related to homogeneous spaces of the general linear Lie group $\GL(n):=\GL(n,\C))$. However the techniques can be easily modified for other Lie groups.

In Section \ref{sec:sp} we study the the complex symplectic group $\Sp(2n)$. This leads us to signed versions of the previous problems. For example we prove the signed Erd\H os--Heilbronn theorem:

\textbf{Theorem \ref{thm:signed-eh}.}
\emph{If $A$ is a set of distinct non zero residue classes $a_1,\dots, a_{n}$ modulo $p$, $a_i + a_j \neq 0$  and $ p > 2k(n-k) + \binom{k+1}{2}$, then:}
\begin{equation*}
 \left| \left\{\sum_{i\in I}\pm a_i\ \big| \ I \subset (1, \dots, n) ,|I| = k\right\}\right | > 2k(n-k) + \binom{k+1}{2}.
\end{equation*}

The main tool of the paper is that we calculate the degree of a torus-invariant projective variety in two different ways: using the Atiyah--Bott--Berline--Vergne integration formula (ABBV formula for short) and the Borel--Hirzebruch formula. We included an appendix with a very short introduction to these formulas.

An interesting aspect of the method is that every lower bound obtained is the dimension of the  homogeneous manifold used for the particular problem.

\section*{Acknowledgement} We thank Gyula K\'arolyi for inspiring discussions on the topic.
\section{Warmup: The Cauchy--Davenport theorem, connections with Lagrange interpolation and the Combinatorial Nullstellensatz}\label{sec:cd}
All the examples in this section can easily be proved by elementary or classical arguments. The section serves as an introduction to the cohomological method we are going to use in the paper. First we reprove the Cauchy--Davenport Theorem using the ABBV formula:

\begin{theorem}
Let $A $ and $B$ be two subsets of the field $\F_p$ with $p$ elements, where $p$ is a prime. Assume that $|A| = r$, $|B| = s$ and $r+s-1 \leq p$, and let $A+B$ denote the sumset
   \[A+B= \{ a +b\ | \ a \in A , b \in B \}.\]
Then $ |A+B| \geq r+s-1$.
\end{theorem}
\begin{proof}

Let $V$ be an $r$-dimensional complex vector space with basis $e_1,\dots, e_r$ and $U$ be an $s$-dimensional complex vector space with basis $f_1, \dots, f_s$.
The group $G_1=\GL(r)$ acts on $V$. Let $T(r)$ denote the compact maximal torus of diagonal matrices  corresponding to the basis $e_1, \dots, e_r$.
Similarly $G_2=\GL(s)$ acts on $U$ and $T(s)$ is the compact maximal torus corresponding to the basis $f_1, \dots, f_s$.
%Let $x_1, \dots, x_r$ denote the Chern roots of $T_1$ and $y_1, \dots, y_s$ the Chern roots of $T_2$.

Consider the map $ U \times V  \to U \otimes V $ defined by the tensor product.
The projectivization of this map is the Segre embedding $ \P(U) \times \P(V ) \to \P(U \otimes V) $.

The Segre variety $S_{r, s}$ is the image of the Segre embedding. It is an invariant subvariety of  the action of the group $\GL(r) \times \GL(s)$ on the space $ \P(U \otimes V) $, in fact it is the unique closed orbit.

By a classical formula the degree of the Segre variety is ${{r+s-2} \choose {r-1}}$. It is classically calculated as an integral:

\begin{equation} \label{eq:deg-segre}
  \deg\big(\P(U) \times \P(V ) \subset \P(U \otimes V)\big)=\binom{r+s-2}{r-1}=\int\limits_{S_{r, s}} c_1^{r+s-2},
\end{equation}

where $c_1$ is the restriction of the first Chern class of the dual of the tautological line bundle $L=\mathcal{O}(1)\to \P(U \otimes V)$.

We can also compute this degree  using equivariant cohomology and the Atiyah--Bott--Berline--Vergne (ABBV for short) integration formula.

\begin{proposition}\cite{atiyah-bott}\label{ab} Suppose that $X$ is a compact manifold and $T$ is a torus acting smoothly on $X$, and the fixed point set $F(X)$ of the $T$-action on $X$ is finite. Then for any cohomology class $x\in H_T^*(X):=H_T^*(X;\C)$
\begin{equation}\label{abformula}\int\limits_X  x=\sum_{f\in F(X)} \frac{x|_f}{e(T_{f}X)}.\end{equation}
Here $e(T_fX)$ is the $T$-equivariant Euler class of the tangent space $T_fX$. The right side is considered in the fraction field of the polynomial ring of $H^*_T=H^*(BT)$ (see more details in \cite{atiyah-bott}): part of the statement is that the denominators cancel when the sum is simplified.
\end{proposition}
First notice that $L$ is a $T(r) \times T(s)$ vector bundle---i.e. the $T(r)\times T(s)$-action lifts to $L$ acting linearly on the fibers---so it has an equivariant first Chern class $\hat c_1$. Because the degree of the integrand is the dimension, any equivariant lift will give the same integral:
\begin{equation} \label{eq:deg-segre-equi}
  \deg\big(\P(U) \times \P(V ) \subset \P(U \otimes V)\big)=\int\limits_{S_{r, s}} \hat c_1^{r+s-2},
\end{equation}
The fixed points of the action of the maximal torus $T(r) \times T(s)$ are the lines spanned by the vectors $e_i \otimes f_j$,
and applying the ABBV formula for $X=S_{r, s}\iso\P^{r-1}\times \P^{s-1}$ we get
\begin{equation}\label{eq:deg-segre2}
 {{r+s-2} \choose {r-1}}  =  \sum_{ i=1}^r \sum_{j=1}^s\frac{(x_i + y_j)^{r+s-2}}{\prod\limits_{k \neq i}(x_i - x_k) \cdot \prod\limits_{l \neq j}(y_i - y_k) },
\end{equation}
where $H^*_{T(r) \times T(s)}=\C[x_1,\dots,x_{r},y_1,\dots,y_{s}]$. We used that the tangent bundle of a projective space is of the form $\Hom (S,Q)$, where $S$ and $Q$ denote the tautological sub- and quotient bundles, respectively.

The integral of a cohomology class of degree less then the real dimension of the manifold on which we integrate is zero. Consequently, for  any set of elements $M$ in any field with $|M|= r+s-2$ the following identity holds:
\begin{equation}\label{eq:deg-segre3}
 {{r+s-2} \choose {r-1}}  =  \sum_{ i=1}^r \sum_{j=1}^s\frac{\prod\limits_{m \in M}(x_i + y_j - m)}{\prod\limits_{k \neq i}(x_i - x_k) \cdot \prod\limits_{l \neq j}(y_j - y_l) }
\end{equation}
\begin{remark}
The idea of adding constants to the linear factors of an identity is a very effective trick used by many authors in this field. We could trace it back to \cite{anr}, but it is also used in \cite{Karasev2012} and \cite{preismann} and several other places. Geometrically this trick corresponds to tensoring the line bundle $L$ with various other line bundles, but we will not explore this connection in this paper.
\end{remark}

If $|A+B| \leq r+s-2 $ then we can choose $M$, such that $ A+B \subset M \subset \F_p$.
Substitute now the elements of the set $A$ into the $x$ variables and the element of the set $B$ into the $y$ variables in the identity (\ref{eq:deg-segre3}).

All terms of the right hand side are divisible by $p$, but the left hand side is  not divisible by $p$, if $r+s-2 \leq p$.

This contradiction proves the theorem.
\end{proof}
\subsection{Cohomological proof of the Coefficient Formula} \label{sec:coeff}
The readers may notice that identity (\ref{eq:deg-segre2}) is a special case of the  Coefficient Formula below, proved and used e.g. in \cite[Lemma 1.1]{karolyi-nagy-petrov}. For the history of the Coefficient Formula see loc. cit.
\begin{lemma}\label{CF} \textbf{(Coefficient Formula)} Let $f\in \F[x_1,\dots,x_n]$ be a polynomial of degree $\deg(f)\leq d_1+d_2+\cdots+d_n$. For arbitrary subsets $C_1,\dots,C_n$
of $\F$ with $|C_i| = d_i + 1$, the coefficient of $\prod x_i^{d_i}$ in $f$ is
\[\sum_{c_ 1\in C_1 }\sum_{c_ 2\in C_ 2} \cdots \sum_{c_n \in C_n } \frac{f(c_1,c_2,\dots,c_n)}
{\phi^{'}_1 (c_1 )\phi^{'}_2 (c_2 )\cdots \phi^{'}_n (c_n )} ,\]
where $\phi_i(z)=\prod_{c\in C_i }(z-c)$.
\end{lemma}

We claim that this formula can be deduced from the ABBV formula, at least for $\F=\C$. Let the manifold $X$ be the product of projective spaces:
  \[X=\bigtimes_{i=1}^n \P^{d_i}.\]
We have the torus $T(d_i+1)$ acting on $\P^{d_i}$ induced by its diagonal action on $\C^{d_i+1}$, so the product of these tori will act on $X$. Let $x_i$ denote the equivariant first Chern class of the dual of the tautological line bundle of $\P^{d_i}$. Applying the ABBV formula for the integral of $f(\alpha_1,\dots,\alpha_n)$, where $\alpha_i$ is the first Chern class of the canonical bundle of $\P^{d_i}$ we obtain the Coefficient Formula by noticing that the integral of a polynomial $f(\alpha_1,\dots,\alpha_n)$ on  $\bigtimes_{i=1}^n \P^{d_i}$ is nothing else, than the coefficient of $\prod \alpha_i^{d_i}$ in $f$.

You can also use the Combinatorial Nullstellensatz of \cite{alon-null} to prove the Cauchy--Davenport theorem. As noted in \cite{karolyi-nagy-petrov}, the Combinatorial Nullstellensatz is an immediate corollary of Lemma \ref{CF}.

To prove the  Cauchy--Davenport theorem you can also refer to the Alon--Nathanson--Ruzsa Lemma \cite[Thm 2.1]{anr}, which is also an immediate corollary of Lemma \ref{CF}.

Summing it up, the Coefficient Formula, the Combinatorial Nullstellensatz and  the Alon--Nathanson--Ruzsa Lemma can be quickly proved using the ABBV formula for the product of projective spaces. Also, in some cases we know that the coefficient in question is positive because of geometric reasons: it is the degree of a projective variety.

In the following sections we show that using other manifolds with torus actions we can arrive to less familiar identities.

\section{Dias da Silva--Hamidoune theorem revisited} \label{sec:eh}
In this section first we reprove the conjecture of Erd\H os and Heilbronn (published in \cite{erdos-graham}) first proved by Dias da Silva and Hamidoune in \cite{Silva-Hamidoune}:

\begin{theorem}\label{thm:eh}
Let $A \subset \mathbb{F}_{p}$ be a set, such that $|A| = n$ and $1 \leq k \leq n$, and let us use the notation
  \[    \bigwedge^kA := \left\{ \sum_{i=1}^{ k}a_i\  | \ a_i \in A ,\ a_i \neq a_j  \text{ if } i \neq j \right\}.  \]
Then  $ |\bigwedge^kA| \geq \min\{(n-k)k+1, p \}$ holds.
\end{theorem}
\begin{proof}
Consider the action of $\GL(n)$ on the alternating $k$-forms $V=\bigwedge^k\C^n$. Then the minimal orbit in $\P V$ is the Grassmannian $\gr_k(\C^n)$ embedded into $\P V$ via the Pl\"ucker embedding.
There are many ways to calculate the degree of this embedding, for example using Schubert calculus and the Pieri rule, as Schubert calculated originally (see also \cite{fulton}) or using the Borel--Hirzebruch formula \cite{borel-hirzebruch2} (see also \cite[sec. 6]{gross-wallach} for details on the degree of the Grassmannian), the answer is the following classical formula:

\begin{equation}\label{eq:deg-grass-bh}
  \deg\big(\gr_k(\C^n)\big)=\int\limits_{\gr_k(\C^n)} \!c_1^{k(n-k)}=\big(k(n-k)\big)!\cdot\prod_{i=1}^k\frac{(i-1)!}{(n-i)!},
\end{equation}
 where $c_1$ is the first Chern class of the dual of the tautological bundle.

The result is in the 0-th cohomology $H^0(\gr_k(\C^n))\iso\Z$. Since we also have $H^0_{T(n)}(\gr_k(\C^n))\iso\C$, and forgetting the action induces isomorphism between these two $\C$'s, we can simply replace the first Chern class with the $T(n)$-equivariant one:
\begin{equation} \label{eq:deg-grass-equi}
  \deg\big(\gr_k(\C^n)\big)=\int\limits_{\gr_k(\C^n)}\! \hat c_1^{k(n-k)}.
\end{equation}
 Recall that the torus $T(n)$ acts on the Grassmannian as restriction of the $\GL(n)$-action to the diagonal matrices, and $\hat c_1$ is the $T(n)$-equivariant first Chern class of the dual of the tautological subbundle over $\gr_k(\C^n)$. The fixed points of the torus action are the coordinate subspaces $V_I:=\langle e_{I_1},\dots,e_{I_k}\rangle$, where $I=(I_1<\cdots<I_k)$ is a $k$-element subset of $\{1,2,\dots,n\}$. We use the notation $I\in \binom nk$.
We apply the ABBV integration formula:

\begin{equation}\label{eq:subst-grass-int}
 \int\limits_{\gr_k(\C^n)}\!\hat c_1^{k(n-k)}=\sum_{J\in \binom{n}{k}}\frac{x_J^{k(n-k)}}{\prod\limits_{i\in J}\prod\limits_{j\notin J} (x_i-x_j)},
\end{equation}
where $H^*_{T(n)}=\C[x_1,\dots,x_n]$ for the $x_i$'s being the \idez{positive} generators and $x_J = \sum_{i \in J}x_i$.

We use again that if the degree of the equivariant cohomology class $\omega$ is smaller than $k(n-k)$---the dimension of  $\gr_k(\C^n)$---then  $ \int\limits_{\gr_k(\C^n)}\!\omega= 0$. Consequently for any set of elements $M$ with $|M|= k(n-k)$ in any field the following identity holds:

\begin{equation} \label{eq:azonossag}
\big(k(n-k)\big)!\cdot\prod_{i=1}^k\frac{(i-1)!}{(n-i)!} =  \sum_{J\in \binom{n}{k}}\frac{\prod\limits_{m\in M}(x_J - m)}{\prod\limits_{i\in J}\prod\limits_{j\notin J} (x_i-x_j)}.
\end{equation}

Assume first that $ p > (n-k)k $, and $ |\bigwedge^kA| \leq (n-k)k $, and choose an $M\subset\F_{p}$ such that  $|M| = (n-k)k $ and  $\bigwedge^kA \subset M $.

If we use the set $M$ in the formula above and substitute the elements of $A$ in the variables $ x_i$, then the right hand side is divisible by $p$, because the numbers $x_J$ are all in $\bigwedge^kA \subset M $.

On the other hand by (\ref{eq:deg-grass-bh}) the left hand side is  not divisible by $p$, if $ p > (n-k)k $.
\begin{remark}
Using an idea we learned from Gyula K\'arolyi we can calculate the degree of the Grasmannian using the localization formula (\ref{eq:azonossag}). Applying the substitution $x_i=i$ and choosing
  \[  M=\left\{\binom{k+1}2+1,\binom{k+1}2+2,\dots,\binom{k+1}2+k(n-k)\right\},  \]
   we can see that all terms in the sum become zero except the one corresponding to $J=\{1,\dots,k\}$, and that term evaluates to
   \[ \frac{\big(k(n-k)\big)!}{\prod\limits_{i=1}^k  \  \prod\limits_{j=k+1}^n (j-i)},\]
   which is easily seen to be equal to the expression for the degree given above.
\end{remark}

Now assume that $ p  \leq (n-k)k $. Let $ B \subset \GL(n)$ be the Borel subgroup consisting of the invertible upper triangular matrices. The Schubert variety $\sigma_I$ is the closure of the Borel orbit $BV_I$ of the coordinate subspace $V_I$. The Schubert variety $\sigma_I$ has dimension $ \sum_ {i=1}^kI_i - \frac{k(k+1)}{2}$. Note that there are other conventions to encode Schubert varieties e.g with partitions $\lambda$ such that its Young diagram fits into a $k\times (n-k)$ rectangle. The translation is given by the \emph{conversion formula}
\begin{equation}\label{conversion}
  I_j=n-k+j-\lambda_j.
\end{equation}

We  choose now a subset $I$ such that the  Schubert variety $\sigma_I$ has dimension $p-1$.

By the  classical formula of Schubert (see also \cite[Ex. 14.7.11]{fulton}) the degree of $\sigma_I$ via the Pl\" ucker embedding is
\begin{theorem}[Schubert]  \label{deg-of-schubert}

   \[\deg(\sigma_I)=\int\limits_{\gr_k(\C^n)} \! c_1^{p-1} \cdot [\sigma_I] =\frac{(p-1)!}{(I_1-1)! \cdots (I_k-1)!} \cdot \prod_{i<j}(I_j -I_i). \]
\end{theorem}

On the other hand we can calculate the degree of the Schubert variety by the ABBV integration formula.
The Schubert variety is  $T$-invariant, so it has  a $T$-equivariant cohomology class $[\sigma_I]_T$  in the Grassmannian manifold. By the same argument as before the integral of an equivariant lift of the integrand will give the same result, so we get the following

\begin{equation}\label{eq:schubertfok}
 \deg(\sigma_I)=\int\limits_{\gr_k(\C^n)} \! \hat c_1^{p-1} \cdot [\sigma_I] =\sum\limits_{J\in \binom{n}{k}}\frac{x_J^{p-1} \cdot [\sigma_I]_T|_J }{\prod\limits_{i\in J}  \ \prod\limits_{j\notin J} (x_i-x_j)},
\end{equation}

where $ [\sigma_I]_T|_J$ is the restriction of the equivariant cohomology class $[\sigma_I]_T$ to the fixed point corresponding to $J$.

We use again that the integral of a class whose degree is less, than the dimension of the Grassmanian, is zero.
This means, that for any set of elements $M$ with $|M|= p-1$ in any field the following identity holds:

\begin{equation} \label{eq:azonossag2}
 \frac{(p-1)!}{(I_1-1)! \cdots (I_k-1)!} \cdot \prod_{i<j}(I_j -I_i) = \sum\limits_{J\in \binom{n}{k}}\frac{\prod\limits_{m\in M}(x_J - m) \cdot [\sigma_I]_T|_J }{\prod\limits_{i\in J}  \   \prod\limits_{j\notin J} (x_i-x_j)} .
\end{equation}

Assuming that $ |A_{k}| \leq p-1 $ we can choose $M$, such that $A_{k} \subset M \subset F_{p}$ and $|M| = p-1$.

If we use this set $M$ in the formula above and substitute the elements of $A$ in the variables $ x_i$, then we get that the right hand side is divisible by $p$, because the numbers $x_I$ are all in $A_{k} \subset M $.

On the other hand, since $I_k\leq n<(n-k)k$, the left hand side is  not divisible by $p$, which gives the desired contradiction, and we finished the proof.
\end{proof}

\begin{remark} The degrees of the Schubert varieties tend to show up in additive combinatorics. A key step in the proof of the Alon--Nathanson--Ruzsa theorem \cite[Prop 1.2]{anr} is identifying these numbers as certain coefficients of a polynomial. The argument above gives a geometric reason for the appearance of these degrees.
\end{remark}

\section{The Sun conjecture}\label{sec:sun}

Z. W. Sun made the following conjecture in \cite{sun=conj} which can be viewed as a linear extension of the Erd\H os--Heilbronn conjecture.

\begin{conjecture}\label{sunconj}
Let $A$ be a  subset of cardinality $n$ of a field $\F$ with characteristic $ p(\F)$ and let $u_1,\dots, u_k \in \F \setminus 0$. If $p(\F) \neq k+1 $ then the following holds:
\begin{equation}
\left|\left\{\sum_{i=1}^k\!u_i a_i \ : \ a_1, \dots, a_k \in A , \text{ and }a_i \neq a_j \text{ if }i\neq j \right\}\right| \geq \min\{p(\F)-\delta, k(n-k)\}
\end{equation}

In the above equation $\delta = 1$, if $n=2$ and $ u_1+ u_2 = 0$, and $\delta = 0$ otherwise.
\end{conjecture}

Recently in \cite{sun-zhao} Z. W. Sun and L. L. Zhao proved the conjecture if $p(\F) \geq \frac{k(3k-5)}{2}. $

We prove a stronger bound in the case $p(\F)$ is big enough.
%, on the other hand the modification of the proof with the Schubert cells does not work directly for the case $p(F)$ is small.

Notice that we do not assume that the $u_i$'s are all different. For example if $u_i=1$ for all $i\leq k$ we get back the Erd\H os--Heilbronn conjecture. Let $\{k_1,\dots,k_t\}$ denote the multiplicities of the $u_i$'s. In particular $\sum k_i=k$.

Using the notation
\[d:= k(n-k) + \sum_{1 \leq i<j \leq t}k_i k_j,\]
 we have the following:

\begin{theorem}\label{sunstrong}
Let $A$ be a  subset of cardinality $n$ of a field $\F$ with characteristic $ p(\F)$ and let $u_1, \dots, u_k \in \F \setminus 0$. If $p(\F) > d $ then the following holds:
\begin{equation}
\left|\left\{\sum_{i=1}^k u_i a_i  : a_1, \dots, a_k \in A , \text{ and }a_i \neq a_j \text{ if }i\neq j\right\}\right| \geq d+1.
\end{equation}
\end{theorem}
\begin{proof}

We again create a polynomial identity using the ABBV formula to prove the theorem.

Consider now the $\GL(n)$-representation $\Gamma_\lambda$ with highest weight
  \[  \lambda=(\mu_{1},\dots, \mu_{1}, \dots, \mu_{t+1}, \dots, \mu_{t+1}), \]
where the integer $\mu_i$ occurs $k_i$ many  times if $i \leq t$
  and $\mu_{t+1}:=0$ occurs $n-k$  times. We also assume that $\mu_1 > \cdots > \mu_{t+1} $.
The minimal orbit of $\P(\Gamma_\lambda)$ is the partial flag manifold $ \fl $, consisting of the flags $ 0 \subset V_1 \subset \cdots  \subset V_t \subset \C^{n} $, where
$ \dim (V_i) = \sum_{j=1}^i k_j$. Notice that $ \dim(V_t) = k$.
The partial flag manifold $ \fl $ is a complex manifold with dimension $d$ and the group $\GL(n)$ acts on it smoothly and transitively.
Let $L$ be the restriction of the dual of the tautological bundle over $\P (\Gamma_\lambda)$ to the flag manifold.
Let us denote the $T(n)$-equivariant first Chern class of $L$ by $\hat c_1$.
The fixed points of the torus action on the flag manifold are exactly the coordinate flags.
If we have a fixed point $F = 0 \subset V_1 \subset \cdots  \subset V_t \subset \C^{n} $, and a permutation $\pi \in S_n$, such that
$V_i = \langle e_{\pi(1)}, \dots,e_{ \pi(\dim(V_i))} \rangle $, then the restriction of $\hat c_1$ to $F$ is
 \[ \hat c_1|_F= \sum\limits_{i=1}^{n}\lambda_{i} x_{\pi(i)} ,\]
where $H^*_{T(n)}=\Z[x_1,\dots,x_n]$ and the $x_i$'s are the \quot{positive} generators.
For a subset $I  \subset \{1, 2, \dots, n\} $ we use the notation $ x_I = \sum_{i \in I}x_i$.

Now if we use the ABBV integration formula we get the following:

\begin{equation}\label{eq:flag-int}
 \deg(\fl) = \int\limits_{\fl}\hat c_1^{d}=\sum_{I_1,\dots, I_{t+1}}\frac{\left(\sum\limits_{j=1}^{t+1} \mu_jx_{I_j}  \right)^d}{\prod\limits_{1\leq r<s \leq t+1}\hspace{2mm}\prod\limits_{j\in I_s} \prod\limits_{i\in I_r}(x_i-x_j)},
\end{equation}
where the summation goes along all partitions of the numbers $1, 2, \dots, n$ into subsets $ I_1,  \dots, I_{t+1} $, where $|I_j| = k_j$.

With the same argument as in the previous cases we can include a new set of variables $b_1,\dots, b_d $:

\begin{equation}\label{eq:subst-flag-int}
 \deg(\fl) = \sum_{I_1,\dots, I_{t+1}}\frac{\prod\limits_{q=1}^{d}\left(\sum\limits_{j=1}^{t+1} \mu_jx_{I_j}  - b_q\right)}{\prod\limits_{1\leq r<s \leq t+1}\hspace{2mm}\prod\limits_{j\in I_s} \prod\limits_{i\in I_r}(x_i-x_j)}
\end{equation}

On the other hand by the Borel--Hirzebruch formula \cite[24.10 Thm]{borel-hirzebruch2} one can easily get that
\begin{equation}\label{bh4flag} \deg(\fl) = d! \cdot \prod_{\lambda_i > \lambda_j } \frac{\lambda_i - \lambda_j}{j-i}. \end{equation}
Gross and Wallach gives a modern introduction to the Borel--Hirzebruch formula in \cite{gross-wallach}, what we found very useful.

We can rewrite \eqref{bh4flag} in terms of the $\mu$'s:
\begin{equation}\label{bh4flaginmu}
\deg(\fl) = d! \cdot \prod\limits_{1\leq a<b \leq t+1 }
       \frac {(\mu_a - \mu_b)^{k_ak_b} }   {\prod\limits_{u=1}^{k_a}\prod\limits_{v=1}^{k_b}(v-u+\sum\limits_{i=a}^{b-1}k_i)}.
           \end{equation}

This means that we have the following polynomial identity in the variables $\mu_i, x_i, b_i $:

\begin{equation}\label{eq:subst-flag-int-id}
 d! \cdot \prod\limits_{1\leq a<b \leq t+1 } \frac{(\mu_a - \mu_b)^{k_ak_b}} {\prod\limits_{u=1}^{k_a}\prod\limits_{v=1}^{k_b}(v-u+\sum\limits_{i=a}^{b-1}k_i)} =
 \sum_{I_1,\dots, I_{t+1}}\frac{\prod\limits_{q=1}^{d}\left(\sum\limits_{j=1}^{t+1} \mu_jx_{I_j}  - b_q\right)}{\prod\limits_{1\leq s<l \leq t+1}\hspace{2mm}\prod\limits_{j\in I_s} \prod\limits_{i\in I_r}(x_i-x_j)}.
\end{equation}

The key observation here is that we can replace the $\mu_i$'s with variables.

Now suppose that our conjecture is false, $p(\F) > d$ and let $|B|=d $ be a subset of $\F$ containing the set
  \[  \left\{\sum_{1\leq i \leq k}\! u_i a_i \ : \ a_1, \dots, a_k \in A , \text{ and }a_i \neq a_j \text{ if }i\neq j\right\}.  \]

Now let us  substitute the elements of $B$ into the variables $b_i$, the elements of $A$ into the variables $x_i$. Into the variables $\mu_i$ we substitute the $u_i$ which had multiplicity $k_i$.

So now from the assumptions we  get that the right hand side is zero, while the left hand side is non zero.
This contradiction proves our theorem.
\end{proof}

Finally we study small primes in the special case $k=n$ and all the $u_i$ are different. Notice that in this case the estimate of the conjecture is trivial, but we show that in fact all elements of the field can be obtained:

\begin{theorem}\label{sun2}
Let $u_1, \dots, u_ n \in \F $ be pairwise distinct elements of a field $\F$ with characteristic $ p$ and let $A=\{a_1,\dots,a_n\}$ be an other   subset of cardinality $n$ of $\F$.
If $n>3$ and $p  \leq \binom{n}2 $ then the following holds:
\begin{equation}
\left|\left\{\sum_{i=1}^{n}  u_i a_{\pi(i)}  : \pi \in S_n  \right\}\right| = p
\end{equation}
\end{theorem}

%Notice that this is a special case of Conjecture \ref{sunconj}: $n>3$ implies that $\delta=0$, The other conditions imply that $p<

 \begin{proof}
Specializing  (\ref{eq:subst-flag-int-id}) we get the identity:

\begin{equation}
V(\lambda_1, \dots, \lambda_n) \cdot \frac{\binom{n}2!}{V(1, 2, \dots, n)}=\sum_{\pi \in S_n}   \frac{\left(\sum\limits_{i=1}^{n}\lambda_i x_{\pi(i)}\right)^{\binom{n}2}}{V(x_{\pi(1)} , \dots , x_{\pi(n)})} ,
\end{equation}
where the $\lambda_i$'s and the $x_i$'s are variables.  Here $V(x_{1} , \dots , x_{n})$ denotes the Vandermonde determinant $\prod\limits_{i<j}(x_j-x_i)$.

Formally you can think of the left hand side as  the degree of the embedding of the full flag manifold to $\P(\Gamma_\lambda)$ for the highest weight $\lambda=(\lambda_1, \dots, \lambda_n)$ with $\lambda_i>\lambda_j$ for $i<j$.

We can assume that $p > \binom{n-1}2$, otherwise we know that $ n \leq p \leq \binom{n-1}2 $, so $n-1 > 3$.
In that case we can fix $\pi(n) = n$ and prove the statement for $n-1$ rather than $n$.

Notice that if $n \geq 5$ then $n < \binom{n-1}2 < p$, so $n < p$, and if $n = 4$, then again $ p \neq 4$, so we have $n < p$.

So assume that $p > \binom{n-1}2$ and let $k = \binom n2 - p + 1$, where we have $k \leq n-1$.
Now  differentiate the identity $k$ times in the variable $\lambda_i$:

\begin{equation}
\sum_{\pi \in S_n}  \frac{x_{\pi(i)}^k \cdot \left(\sum\limits_{i=1}^{n}\lambda_i x_{\pi(i)}\right)^{p-1}}{V(x_{\pi(1)} , \dots , x_{\pi(n)})} =  \frac{\partial^k}{\partial \lambda_i^k}  V(\lambda_1, \dots, \lambda_n) \cdot \frac{(p-1)!}{V(1, 2, \dots, n)}.
\end{equation}

From these identities we get:

\begin{equation}
\sum_{\pi \in S_n} \frac{\left(\sum\limits_{i=1}^{n}(\lambda_i x_{\pi(i)})^k\right) \cdot \left(\sum\limits_{i=1}^{n}\lambda_i x_{\pi(i)}\right)^{p-1}}{V(x_{\pi(1)} , \dots , x_{\pi(n)})}  =
\sum_{i=1}^{n}  \left( \lambda_i^k  \cdot \frac{\partial^k}{\partial \lambda_i^k} V(\lambda_1, \dots, \lambda_n) \cdot \frac{(p-1)!}{V(1, 2, \dots, n)} \right) .
\end{equation}

Now we have the following easy identity for $k \leq n-1$ (use that the left hand side is also alternating):

\begin{equation}
\sum_{i=1}^{n}  \left( \lambda_i^k  \cdot \frac{\partial^k}{\partial \lambda_i^k} V(\lambda_1, \dots,\lambda_n) \right) = V(\lambda_1, \dots,\lambda_n) \cdot \binom n{k+1}  k!
\end{equation}

implying that

\begin{equation}
\sum_{\pi \in S_n} \frac{\left(\sum\limits_{i=1}^{n}(\lambda_i x_{\pi(i)})^k\right) \cdot \left(\sum\limits_{i=1}^{n}\lambda_i x_{\pi(i)}\right)^{p-1}}{V(x_{\pi(1)} , \dots , x_{\pi(n)})}  = V(\lambda_1, \dots, \lambda_n) \cdot \binom n{k+1}  k! \cdot \frac{(p-1)!}{V(1, 2, \dots, n)}
\end{equation}

Again we can include a set $B$ of constants with $|B| = p-1$ without changing the value of the formula:

\begin{equation} \label{eq:sun-different}
\sum_{\pi \in S_n} \frac{\left(\sum\limits_{i=1}^{n}(\lambda_i x_{\pi(i)})^k\right) \cdot \prod\limits_{b\in B}\left( \sum\limits_{i=1}^{n}
\lambda_i x_{\pi(i)}-b \right)}{V(x_{\pi(1)} , \dots , x_{\pi(n)})}  = V(\lambda_1, \dots, \lambda_n) \cdot \binom n{k+1}  k! \cdot \frac{(p-1)!}{V(1, 2, \dots, n)}
\end{equation}

Assume that we have a counterexample $A$ and numbers $u_1, \dots, u_ n \in \F $, then we have a set $ B \subset \F$, such that $|B| = p-1$ and
  \[ \left\{\sum\limits_{i=1}^{n}u_i a_{\pi(i)}  : \pi \in S_n \right\} \subset B.  \]
Now in the identity \eqref{eq:sun-different} we substitute the elements $a_i$ into the  variables $x_i$, the elements $u_i$ into the variables $\lambda_i$. We have that the left hand side is $0$ in $\F$, while the right hand side is not, which is a contradiction.
\end{proof}

\section{The Grasshopper: a simultaneous generalization of the Erd\H os--Heilbronn problem} \label{sec:grasshopper}

On the  $50$-th International Mathematics Olympiad for highschool students the following problem was given:

\textbf{Imo 2009/6 \cite{khramtsov}:} \emph{Let $a_1, \dots, a_n$ be distinct positive integers and let $M$ be a set of $n-1$ positive integers not containing $s = \sum_{i=1}^{n}a_i$. A grasshopper is to jump along the real axis, starting at the point $0$ and making $n$ jumps to the right with lengths $a_1, \dots , a_n$ in some order. Prove that the order can be chosen in such a way that the grasshopper never lands on any point in $M$.}
\medskip

The problem was said one of the hardest problems on IMO until that time, since only $3$ people could solve it. The proof is an inductive combinatorial proof and deeply depends on the fact that the numbers $a_1, \dots, a_n$ are positive.

G\'eza K\'os suggested the following form in the Mathematical  and Physical Journal for Secondary Schools \cite{komal}:

\textbf{K\"omal: A. 496.} \emph{Let $a_1,a_2,\dots,a_{k}$ be distinct integers for $k=2n$ and let $M$ be a set of $n$ integers not containing 0 and $s=a_1+a_2+\cdots+a_{2n}$. A grasshopper is to jump along the real axis, starting at the point 0 and making ${k}$ jumps with lengths $a_1,a_2,\dots,a_{2n}$ in some order. If $a_i>0$ then the grasshopper jumps to the right; while if $a_i<0$ then the grasshopper jumps to the left, to the point in the distance $|a_i|$ in the respective steps. Prove that the order can be chosen in such a way that the grasshopper never lands on any point in $M$.}
\medskip

This form was proved by the second author using the Combinatorial Nullstellensatz, however  to prove that a certain coefficient of a polynomial is nonzero was cumbersome.

In \cite{kg=grasshopper}  the positivity of this coefficient is proved by an inductive tricky way.
In this section we will identify this coefficient with an intersection number  in a flag manifold.

The argument  of \cite{kg=grasshopper} in fact gives more: at every step we can give a different forbidden sets of integers $M_i$ for $i=1,\dots,2n-1$ of cardinality $n$, and the grasshopper will be able to avoid the positions in $M_i$ with his $i$-th jump. We would like to generalize this version: characterizing the sequences of non negative integers $\mathbf{b}=(b_1,\dots,b_{k-1})$ for which the grasshopper can avoid forbidden sets of integers $M_i$ for $i=1,\dots,k-1$ of cardinality $b_i$. We will call these sequences $\mathbf{b}$ \emph{admissible}:
\begin{definition}\label{de:admissible}
A sequence of non negative integers $\mathbf{b}=(b_1,\dots,b_{k-1})$ is admissible if for all $M_i\subset\Z$ and $|M_i|=b_i$, and for all $a_1,a_2,\dots,a_{k}$  distinct integers there is a permutation $\pi\in S_k$ such that for all $j=1,\dots k-1$
\[ \sum_{i=1}^{j}a_{\pi(i)}\notin M_i.\]
\end{definition}
Then  \textbf{A. 496.} is equivalent with the admissibility of $\mathbf{b}=(n,\dots,n)$ for $k=2n$.

Notice that we no longer assume that $k$ is even. In this section we characterize the admissible sequences:

\begin{theorem} \label{szocsalt-yes} A sequence $\mathbf{b}$ is admissible if and only if it satisfies the following system of linear inequalities:
for any subset $P \subset \{1, \dots ,k-1 \}$ the  condition
\begin{equation}\label{hall-system}
   \sum\limits_{p \in P}b_p\ \ \leq \ \ \Big|\{(i,j):\exists\, p\in P, \text{ such that } 1\leq i \leq p \leq j\leq k-1\}\Big|
\end{equation}
holds.
\end{theorem}

You can interpret the right hand side as the number of 'intervals' of $\{1, \dots ,k-1 \}$ having nonempty intersection with $P$. This theorem tells us that it is enough to check $2^{k-1}$ simple inequalities to decide whether $\mathbf{b}$ is admissible.
\begin{remark} \label{rem:grasshopper=>eh} An easy calculation shows that for sequences $\mathbf{b}$ with only one non zero $b_i$ Theorem \ref{szocsalt-yes} specializes to the Erd\H os--Heilbronn theorem in 0 characteristics. So in the Erd\H os--Heilbronn problem there is only one forbidden set, as in the Grasshopper problem there are several. This way one can think of the Grasshopper problem as a simultaneous Erd\H os--Heilbronn problem.
\end{remark}

First we create an identity by calculating the degree of the full flag manifold in two different ways similarly to the proof of Theorem \ref{sunstrong}.

Let $\Gamma_\lambda$ be the irreducible representation of $\GL(k)$ with highest weight $\lambda=(\lambda_1>\lambda_2>\cdots>\lambda_{k-1}>\lambda_{k}=0)$. This weight is in the open Weyl chamber, so the minimal orbit in $\P(\Gamma_\lambda)$ is the complete flag variety $\fl=\fl(k)$. Its degree in $\P(\Gamma_\lambda)$ is calculated by the integral
  \[  \deg \big(\fl\subset \P(\Gamma_\lambda)\big)=\int\limits_{\fl}\hat c_1(L_\lambda)^d, \]
where $\hat c_1(L_\lambda)$ is the torus equivariant first Chern class of the canonical line bundle $L_\lambda$ and $d=\binom k2=\dim(\fl)$. Using the ABBV formula to calculate the integral we get

\begin{equation}\label{eq:deg-fl-with-ab}
 \deg \big(\fl\subset \P(\Gamma_\lambda)\big)=\sum_{\pi\in S_k}\frac{\left(\sum\limits_{i=1}^k\lambda_ix_{\pi(i)}\right)^d}{e(T_\pi\fl)},
\end{equation}
where $e(T_\pi\fl)=(-1)^{\binom k2} \sgn(\pi)V(\mathbf{x})$.

We can also use the Borel--Hirzebruch formula, which gives

\begin{equation}\label{eq:deg-fl-with-bh}
 \deg \big(\fl\subset \P(\Gamma_\lambda)\big)=(-1)^{\binom k2}\binom d\Delta V(\lambda),
\end{equation}
where $ \Delta = (0, 1, \dots k-1)$.

Now we arrive to the identity:
\begin{equation}\label{eq:deg-fl-bh=ab}
(-1)^{\binom k2}\binom d\Delta V(\lambda)=\sum_{\pi\in S_k}\frac{\left(\sum\limits_{i=1}^k\lambda_ix_{\pi(i)}\right)^d}{e(T_\pi\fl)}.
\end{equation}
Notice that both sides are polynomials in the variables $\lambda_i$, and the identity is valid for integer substitutions in an open cone, therefore these polynomials agree.

\begin{remark}The actual value $e(T_\pi\fl)=(-1)^{\binom k2} \sgn(\pi)V(\mathbf{x})$ of the denominator is not used in the following, however we can see that \eqref{eq:deg-fl-bh=ab} is equivalent to
   \[\binom d\Delta V(\lambda) V(\mathbf{x})  =   \alt_\mathbf{x}\big((\lambda,\mathbf{x})^d\big),  \]
where $(\lambda,\mathbf{x})$ denotes the scalar product of the vectors $\lambda$ and $\mathbf{x}$ and $\alt_\mathbf{x}$ refers to the antisymmetrization with respect to the $\mathbf{x}$ variables. This formula also can be proved directly, without using geometry.
\end{remark}
We define the homogenous polynomial $P$ in the variables $v_i$ with degree $d$ by substituting $ \lambda_i = \sum\limits_{j=i}^{k}v_{j}$ into \eqref{eq:deg-fl-bh=ab}, notice that the polynomial $P$ does not depend on the variable $v_k$.

The coefficient of $v^\mathbf{b}$ for $\mathbf{b}=(b_1,\dots,b_{k-1})$ is:
\begin{equation}
\coef(P,v^\mathbf{b})= \binom{d}{\mathbf{b}} \sum\limits_{\pi\in S_k }
\frac{ \prod\limits_{i=1}^{k-1}\left(\sum\limits_{j=1}^{i} x_{\pi(j)}\right)^{b_i}}{e(T_{\pi}\fl)},
\end{equation}
where $\binom{d}{\mathbf{b}}$ denotes the multinomial coefficient.

By the usual trick we can add lower order terms:
\begin{equation}
\coef(P,v^\mathbf{b})= \binom{d}{\mathbf{b}} \sum\limits_{\pi\in S_k } \frac{ \prod\limits_{m \in M_i }\left(\sum\limits_{j=1}^{i} x_{\pi(j)}
   -m \right)}{e(T_{\pi}\fl)},
\end{equation}
for given sets of numbers $M_i$ with $|M_i|=b_i$. This implies that if the coefficient $\coef(P,v^\mathbf{b})$ is not zero, then $\mathbf{b}$ is admissible. Notice that the actual form of the denominators is not important, only the property that they are not zero if we substitute different numbers into the $x_i$'s.

So we have the following expression:
\begin{equation} \label{eq:mu}
\frac{\binom d\Delta}{\binom d{\mathbf{b}}}\, \mu(\mathbf{b})  =   \sum\limits_{\pi\in S_k } \frac{ \prod\limits_{m \in M_i }\left(\sum\limits_{j=1}^{i} x_{\pi(j)}
   -m \right)}{e(T_{\pi}\fl)},
\end{equation}
where $\mu(\mathbf{b})=(-1)^{\binom k2}\coef(V(\lambda),v^\mathbf{b})$.

Using a formula of Duan in \cite{duan} for the degree of the flag manifold one can give a closed formula for $\mu(\mathbf{b})=\coef(V(\lambda),v^\mathbf{b})$, however this formula is a sum of terms with different signs, and it is not clear from it which coefficients are zero. So we try to decide which $\mu(\mathbf{b})$'s are non zero without calculating their value.

The key observation is that because of $\lambda_i-\lambda_j=v_{i}+\dots+v_{j-1}$,  the positivity of the coefficient of $v^\mathbf{b}$ in the product is equivalent to the existence of a  matching covering the lower class of the bipartite graph $\mathcal{B}_\mathbf{b}$ defined in \ref{de:bipartite}:

\begin{definition}\label{de:bipartite}
We assign a bipartite graph $\mathcal{B}_\mathbf{b}=(U,D_\mathbf{b},E_\mathbf{b})$ to a sequence of non negative integers $\mathbf{b}=(b_1,\dots,b_{k-1})$. The upper  class  $U$ consists of the the  pairs $(j,l)$ with $1 \leq j \leq l \leq k-1 $ independently of $\mathbf{b}$. The lower class $D_\mathbf{b}$ consists of the pairs $ (i, t) $, where $1 \leq i \leq k-1 $ and $ 1 \leq t \leq b_i$. $E_\mathbf{b}$ is the set of edges: there is an edge between $(j, l)$ and $(i, t)$ if $j \leq i \leq l $.

We say that the sequence $\mathbf{b}$ is a \emph{matching sequence} if $\mathcal{B}_\mathbf{b}$ has a matching covering the lower class $D_\mathbf{b}$.
\end{definition}

The elements of $U$ correspond to the factors of $V(\lambda)$ and the elements of $D_\mathbf{b}$ correspond to the factors of $v^\mathbf{b}$. Note that $\mu(\mathbf{b})>0$ implies that $|\mathbf{b}|:=\sum_i b_i=\binom k2$, so in these cases a matching of $D_\mathbf{b}$ is a perfect matching of the bipartite graph $\mathcal{B}_\mathbf{b}$. We will call these $\mathbf{b}$'s \emph{perfect matching sequences}. However we also interested in sequences with $|\mathbf{b}|<\binom k2$, but these cases can be reduced to the $|\mathbf{b}|=\binom k2$ case (the choice of $a_i:=i$ shows that for any admissible sequence $|\mathbf{b}|\leq\binom k2$ holds):
The structure of these bipartite graphs is quite simple, because the set of neighbours of a vertex $(i,t)\in D_\mathbf{b}$ is independent of $t$. Consequently a simple combinatorial argument gives the following.
\begin{lemma} The sequence $\mathbf{b}$ is matching if and only if there is a matching sequence $\bar{\mathbf{b}}=(\bar b_1,\dots,\bar b_{k-1})$ with $|\bar{\mathbf{b}}|=\binom k2$ dominating $\mathbf{b}$, i.e. $\bar b_i\geq b_i$ for $i=1,\dots,k-1$. \qed
\end{lemma}

The definition immediately implies that any sequence of non negative integers dominated by an admissible sequence is also admissible, therefore, we showed so far (by using the standard substitution trick into the right hand side of \eqref{eq:mu}) that if $\mathbf{b}$ is a matching sequence then it is also admissible. To finish one direction of the proof of Theorem \ref{szocsalt-yes} we  need to prove that

  \begin{proposition}\label{matchinginequalities}
A sequence of non negative integers $\mathbf{b} = (b_1, \dots, b_{k-1} )$ is matching if and only if $\mathbf{b}$ satisfies the following system of linear inequalities:
for any subset $P \subset \{1, \dots ,k-1 \}$ the  condition
      \begin{equation}\label{hall-system}   b_P\leq K(P)  \end{equation}
 holds, where $b_P=\sum\limits_{p \in P}b_p$ and $K(P)$ is the number of pairs $(i,j)$ such that $ 1 \leq i \leq j \leq k-1$ and there exists a $p \in P$ for which $i \leq p \leq j$.
  \end{proposition}
\begin{proof} We use the Hall theorem \cite{hall1935representatives}. Using again the fact that the set of neighbors of a vertex $(i,t)\in D_\mathbf{b}$ is independent of $t$, we can see that it is enough to check the Hall condition only for the subsets $H$ of $D_\mathbf{b}$ which have the property that if $(i,t)\in H$, then all vertices of the form $(i,s)$ are also in $H$. For these subsets the Hall condition gives exactly the inequalities (\ref{hall-system}).
\end{proof}
For the other direction of Theorem \ref{szocsalt-yes} we need the following.
    \begin{proposition} Admissible sequences are matching.  \end{proposition}
\begin{proof}
In the proof we use the notation $\binom{k}2 = [k]_2$ for typographical reasons.
Let $\mathbf{b}=(b_1,\dots,b_{k-1})$ be an admissible sequence. We choose $a_i=i$ for $i=1,\dots,k$. For a given $P=\{p_1,\dots,p_m\}\subset  \{1, \dots ,k-1 \}$  of Proposition \ref{matchinginequalities} we choose the forbidden sets $M_{p_i}$ to be  intervals of integers of length $b_{p_i}$ in such a way that the grasshopper is forced to jump to the right of $M_{p_i}$ in the $p_i$-th second. Admissibility of $\mathbf{b}$ implies the existence of a permutation $\pi\in S_k$ such that \[s_j:=\sum\limits_{i=1}^{j}a_{\pi(i)}= \sum\limits_{i=1}^{j}{\pi(i)} \notin M_j.\]
We call such a $\pi$ an allowed permutation. For $p_1$ we have $s_{p_1}=\sum\limits_{i=1}^{p_1}\pi(i)\geq \sum\limits_{i=1}^{p_1}i=[p_1 +1]_2$. Therefore for the choice
\[M_{p_1}=\left[[p_1 +1]_2, [p_1 +1]_2 +b_{p_1}-1\right] \]
the grasshopper in the $p_1$-th second must be on the right of $M_{p_1}$, i.e.
\[ s_{p_1}=\sum\limits_{i=1}^{p_1}\pi(i)\ \geq \ [p_1 +1]_2+b_{p_1} \]
for any allowed permutation $\pi$.

In the $p_2$-th second the grasshopper must be further to the right by at least $\sum\limits_{j=1}^{p_2-p_1}i= [p_2 - p_1 +1]_2$:
\[ s_{p_2}=\sum\limits_{i=1}^{p_2}\pi(i)\ \geq \ [p_1 +1]_2+b_{p_1}+[p_2 - p_1 +1]_2 . \]
If we choose
\[M_{p_2}=\big[[p_1 +1]_2+b_{p_1}+ [p_2 - p_1 +1]_2,[p_1 +1]_2+b_{p_1}+ [p_2 - p_1 +1]_2+b_{p_2}-1\big], \]
then we assured that
\[ s_{p_2} \geq [p_1 +1]_2+b_{p_1}+ [p_2 - p_1 +1]_2+b_{p_2}.\]
Continuing with the same strategy we choose $M_{p_i}=[x_i,x_i+b_{p_i}-1]$ for
\[x_i=\sum\limits_{j=0}^{i-1}[n_j +1]_2+\sum\limits_{j=1}^{i-1}b_{p_j},  \]
where $n_0=p_1$ and $n_j=p_{j+1}-p_j$ for $j=1,\dots,m$. By induction we can see that
\[ s_{p_m}\geq x_m+b_{p_m}.\]
For the last $k-p_m$ jumps the grasshopper moves again at least $\sum\limits_{i=1}^{k-p_m}i=[k-p_m+1]_2$ to the right:
\[s_k=[k+1]_2 \geq x_m+b_{p_m}+[k-p_m+1]_2.\]
By definition
\[K(P)=[k]_2-[k-p_m]_2-\sum\limits_{j=0}^{m-1}[n_j]_2,\]
and
\[[k]_2-[k-p_m]_2-\sum\limits_{j=0}^{m-1}[n_j]_2=[k+1]_2-[k-p_m+1]_2-\sum\limits_{j=0}^{m-1}[n_j +1]_2,\]
since $(k-p_m)+\sum\limits_{j=0}^{m-1}n_j=k$. This implies that
$b_P\leq K(P)$, if there is a permutation $\pi\in S_k$ for which the grasshopper avoids the forbidden sets, and Proposition \ref{matchinginequalities}. implies that $\mathbf{b}$ is matching.
\end{proof}
And we finished the proof of Theorem \ref{szocsalt-yes}.\qed
\bigskip

For perfect matching sequences the system of inequalities (\ref{hall-system}) can be replaced by a simpler one:
\begin{proposition} \label{hall} A sequence $\mathbf{b}$ with $|\mathbf{b}|=\binom k2$ is a perfect matching, i.e.
the coefficient $\mu(\mathbf{b})=\coef(V(\lambda),v^\mathbf{b})$ is non-zero if and only if
%\begin{definition} A sequence of non negative integers $\mathbf{b}=(b_1,\dots,b_{k-1})$ is \emph{tame} if
%$|\mathbf{b}|=\sum\limits_{i=1}^{k-1} b_i = \binom k2$ and
%   \[\sum_{j= s}^t b_j \geq \binom{t - s +2}2  \text{ for every }1 \leq s \leq t \leq k-1.\]
%\end{definition}
\begin{equation}\label{perfect-system}
\sum_{j= s}^t b_j \geq \binom{t - s +2}2  \text{ for every }1 \leq s \leq t \leq k-1.
\end{equation}
\end{proposition}
\begin{proof} Now the two partitions of the bipartite graph $\mathcal{B}_\mathbf{b}$ have the same size, so we can check the Hall condition in the other direction:
Let $ B $ be a subset of $U$ and let
   \[\Sh(B) = \{  a \ | \ 1\leq a \leq k-1 ,\ \exists (j,l)\in B,\  j \leq a \leq l \} \]

be the \quot{shadow} of $B$. Then the size of the neighborhood of $B$ is $ \sum_{a \in \Sh(B)}b_a$.
Let us write $ \Sh(B) = A_1 \cup \cdots\cup A_s$ as  disjoint union of maximal intervals of integers.

The largest $B$ with the same shadow has  size $ \sum_{i=1}^{s}\binom{|A_i|+1}{2}$ so the Hall condition  is equivalent to the system of inequalities
\begin{equation}\label{shadow-ineqs}
  \sum_{a \in \Sh(B)}b_a \geq \sum_{i=1}^{s}\binom{|A_i|+1}{2} \ \ \text{for all possible shadows.}
\end{equation}
(\ref{perfect-system}) is a subsystem of this system of inequalities: the ones where the shadow is a single interval. Taking the sum of inequalities for all the intervals $A_i$ we can see that (\ref{shadow-ineqs}) is equivalent to  (\ref{perfect-system}).
\end{proof}

\begin{remark}\label{rem:c1eibi}
Let $E_i$ denote the rank $i$ tautological subbundle over the full flag manifold and $|\mathbf{b}|=\binom k2$. Then the ABBV formula gives
\begin{equation}
\int\limits_{\fl} \prod_{i=1}^{k-1}c_1(E^*_i)^{b_i}=
\sum\limits_{\pi\in S_k }
\frac{ \prod\limits_{i=1}^{k-1}\left(\sum\limits_{j=1}^{i} x_{\pi(j)}\right)^{b_i}}{e(T_{\pi}\fl)}.
\end{equation}
By the equations \eqref{eq:deg-fl-bh=ab}--\eqref{eq:mu} we can identify this integral with the coefficient
\[ \frac{\binom d\Delta}{\binom d{\mathbf{b}}}\, \mu(\mathbf{b}).\]
The cohomology ring $H^*(\fl)$ of the flag manifold is generated by the classes $c_1(E_i^*)$. Therefore the previous proof and  Poincar\' e duality implies that $\mathbf{b}$ is a matching sequence if and only if $\prod_{i=1}^{k-1}c_1(E^*_i)^{b_i}$ is not zero in $H^*(\fl)$.

 The bundles $E^*_i$ are globally generated, which implies the non negativity of the coefficients $\mu(\mathbf{b})$.

%It is clear from the proof that the coefficients $K_\mathbf{b}$ have a cohomological interpretation. For  sequences $|\mathbf{b}|=\binom k2$ we have
%\begin{equation}
%K_\mathbf{b} = \int\limits_{\fl} \prod_{i=1}^{k-1}c_1(E^*_i)^{b_i},
%\end{equation}
%where $E_i$ is the rank $i$ tautological subbundle over the full flag manifold.
%

\end{remark}
\begin{remark} \label{coef-mod-p} It would be interesting to have a Grasshopper Theorem  in characteristic $p$. G\'eza K\'os mentions in \cite{kg=grasshopper} that the coefficients $\mu(\mathbf{b})$ can have large prime factors. Therefore it is not clear how this mod $p$ Grasshopper Theorem should look like.
\end{remark}

\subsection{Schubert varieties in the flag manifold and Bruhat restrictions for the grasshopper}

%In this section we show how to use the equivariant topology of the flag manifold to get other combinatorial statements.
%Let's have a vector space $V$ with distinguished basis $ e_1, \dots , e_k$ and the full flag manifold $\fl$ consisting of subspace collections $0 < V_1 < V_2 <\cdots<V_k = V$,
%where $\dim_{\C}(V_i)=i$.
It is intuitively clear that it is easier for the grasshopper to avoid the forbidden positions if there are less then $\binom k2$ of them.
In this section  we show that if $|\mathbf{b}|< \binom k2$ then the grasshopper does not have to use all the permutations in $S_k$.

The standard action of the linear group $\GL(k)$ on $\C^k$ induces an action on the flag manifold $\fl$.
Let us denote the subgroup of $\GL(k)$ consisting of diagonal matrices by $T$ and upper triangular matrices by $B$.
For a permutation $w \in S_k$ let $F_w $ be the coordinate flag, for which $ V_i = \langle e_{w(1)}, \dots, e_{w(i)} \rangle $.

%We will need the Bruhat order of the permutations in $S_k$:

For a permutation $w$ let $l(w)$ denote the number of inversions in it.
If $w_1$ and $w_2$ are two permutations and $ w_2 = s(i, j) w_1$ where $s(i, j)$ is a transposition and $l(w_2) = l(w_1) + 1$, then we say that $w_2$ covers $w_1$ and $w_2 > w_1$.
The Bruhat order on  $S_k$ is the transitive and reflexive closure of this relation.
The hierarchy of the $B$-orbits is governed by the Bruhat order  (see e.g. \cite{brion-flag}):

\begin{proposition}\label{schubertcell}\mbox{}

\begin{enumerate}[a)]
  \item The fixed points of $T$ in $\fl$ are the coordinate flags $F_w$, where $w \in S_k$.
  \item $\fl$ is the disjoint union of the orbits $ C_w := B F_w $, where $w \in S_k$.
  \item Let $X_w$ be the Zariski closure of  $C_w$ in $\fl$, then we have $\dim(X_w) = l(w)$ and:
\begin{equation}
X_w = \bigcup\limits_{v \in S_k,\ v \leq w}C_v,
\end{equation}
where $v \leq w$ means the relation in the Bruhat order.
\end{enumerate}

\end{proposition}

The varieties $X_w$ are invariant for the $T$-action, their $T$-invariant cohomology classes in $\fl$ will be denoted by $[X_w]$.

For $\lambda_1>\lambda_2>\cdots>\lambda_k = 0$ the flag manifold $ \fl $  is $T$-equivariantly embedded into $\P(\Gamma_\lambda)$, where the degree of the Schubert variety $X_w$  is given by the following integral:

\begin{equation}\label{deg-of-xw}
  \deg \big(X_w\subset \P(\Gamma_\lambda)\big)=\int\limits_{\fl}\hat c_1(L)^{l(w)} \cdot [X_w].
\end{equation}

Notice that this degree is a polynomial $P_w(\lambda)$ in the variables $\lambda=(\lambda_1, \lambda_2, \dots, \lambda_{k-1})$.

The ABBV integration formula then implies following:

\begin{equation}
P_w(\lambda) = \sum_{u \in S_{k}} \frac{[X_w]|_{F_{u}} \cdot \left(\sum\limits_{i=1}^{k} x_{u(i)} \lambda_{i}\right)^{l(w)}}{\sgn(u) \cdot V(\mathbf{x})}
\end{equation}

Proposition \ref{schubertcell}. implies that the variety $ X_w$ contains exactly the torus fixed points $F_{u}$, for which $u \leq w $ in the Bruhat order.
The class $[X_w]$ is supported on $X_w$, implying that if $ F_{u} \notin X_w $, then $ [X_w]|_{F_{u}} = 0$.
Consequently only the terms with $u \leq w$ contribute:

\begin{equation}
P_w(\lambda) = \sum_{u \leq w} \frac{[X_w]|_{F_{u}} \cdot \left(\sum\limits_{i=1}^{k} x_{u(i)} \lambda_{i}\right)^{l(w)}}{\sgn(u) \cdot V(\mathbf{x})}.
\end{equation}

Now we use again the change of variables:
\[ \lambda_i=v_i+\cdots+v_{k-1}, \ \ i=1,\dots ,k-1.\]

With this notation let us denote $ P_w(\lambda) = R_w(\mathbf{v})$ and let $\mathbf{b}=(b_1,\dots,b_{k-1})$ be a sequence with $|\mathbf{b}| = l(w)$.
By the very same calculation as before we get:

\begin{equation}
\coef(R_w, v^\mathbf{b})= \sum_{u \leq w} \frac{[X_w]|_{F_{u}} \cdot \prod\limits_{i=1}^{k-1} \left(\sum\limits_{j=1}^{i}x_{u(j)}\right)^{b_i}}{\sgn(u) \cdot V(\mathbf{x})}
\end{equation}

Notice that the number on the right hand side is an equivariant integral on the flag manifold of a class with degree equal to the dimension $\binom{k}{2}$, so it is a constant.

\begin{remark}\label{pairing}
  If we denote the ordinary cohomology class of $X_w$ by $[X_w]$ too, then the right hand side is equal to $ \int\limits_{\fl} [X_w] \cdot \prod_{i=1}^{k-1} c_1(E_i^*)^{b_i} $, which is non negative, since the bundles $E_i^*$ are globally generated.
\end{remark}

As before, for any sets $M_i$ of integers with $|M_i|  = b_i$, $i=1,\dots,k-1$ we have the following identity:

\begin{equation}
\coef(R_w, v^\mathbf{b})= \sum_{u \leq w} \frac{[X_w]|_{F_{u}} \cdot \prod\limits_{i=1}^{k-1}   \  \prod\limits_{m \in M_i} \left(\sum\limits_{j=1}^{i}x_{u(j)} - m \right)}{\sgn(u) \cdot V(\mathbf{x})},
\end{equation}

implying the following theorem:

\begin{theorem}\label{szocsalt2}
Let $a_1,a_2,\dots,a_{k}$ be distinct integers, $w \in S_k$ and let $M_1, \dots, M_{k-1}$ be sets of integers with $|M_i| = b_i$ not containing 0 and $s=a_1+a_2+\cdots+a_{k}$.
Let $w\in S_k$ be a permutation with $|\mathbf{b}| = l(w)$.
If $ \coef(R_w, v^\mathbf{b}) \neq 0$, then there is a permutation $\pi \leq w $, such that $\sum\limits_{j=1}^{i} a_{\pi(j)}  \notin M_i $ for all $1\leq i\leq k-1$. \qed
\end{theorem}
There are plenty of such $w, \mathbf{b}$ pairs:
\begin{proposition}\label{josorletezes}\mbox{}

\begin{enumerate}
    \item If $w$ is an arbitrary permutation, then there exists a $\mathbf{b} = (b_1, \dots, b_{k-1} ) $ with $|\mathbf{b}|=\sum\limits_{i=1}^{k-1}b_i = l(w)$, such that $\coef(R_w, v^\mathbf{b}) \neq 0$.
    \item If  $\mathbf{b} = (b_1, \dots, b_{k-1} ) $ is a matching sequence, then there exists a permutation $w$ with $|\mathbf{b}|=\sum\limits_{i=1}^{k-1}b_i = l(w)$ such that $\coef(R_w, v^\mathbf{b}) \neq 0$.
\end{enumerate}

\end{proposition}
\begin{proof} Fix $m\leq\binom k2$. The classes $[X_w]$ with $l(w)=m$ generate the cohomology group $H^{2m}(\fl)$. Also, by Remark \ref{rem:c1eibi}, the classes $\prod_{i=1}^{k-1} c_1(E_i^*)^{b_i} $ with $\sum\limits_{i=1}^{k-1}b_i =m$ generate the cohomology group $H^{2(\binom k2-m)}(\fl)$. Therefore Poincar\'e duality and Remark \ref{pairing}. implies the proposition.
\end{proof}
We can give a combinatorial description for the condition $\coef(R_w, v^\mathbf{b}) \neq 0$.

We say that $\id =u_0< u_1 < \cdots < u_{l(w)} = w$ is a Bruhat chain, if $u_{i+1}$ covers $ u_{i}$ for $i=0,1,\dots,l(w)-1$, and we denote the set of Bruhat chains from $\id$ to $w$ by $C$.

For a cover $u_{i+1} = s(n, m) u_{i}$ let us define $ A(u_i <u_{i+1}) = \sum_{t=n}^{m-1}v_t$ and say that $v_t$ is compatible with the cover if $v_t$ occurs in $ A(u_i <u_{i+1})$.
For a Bruhat chain $ c = (\id =u_0< u_1 < \cdots < u_{l(w)} = w)$ we use the notation  $A(c) = \prod\limits_{i=0}^{ l(w)-1}A(u_i <u_{i+1})$.

Postnikov and Stanley gives the  following description for the polynomial $R_w$:

\begin{proposition}\label{shubfok}\cite{postnikov-stanley}  With the notations above we have
   \[R_w = \sum_{c \in C}A(c).  \]
\end{proposition}

This implies that

\begin{corollary}
$\coef(R_w, v^\mathbf{b}) \neq 0$ if and only if there exists a Bruhat chain $ c = \id < u_1 < \cdots < u_{l(w)} = w$, and variables $v_{x(i)}$ for $0\leq i \leq l(w)-1$, such that $v_{x(i)}$ is
compatible with the cover $u_i <u_{i+1}$ and $\prod\limits_{i=0}^{ l(w)-1} v_{x(i)} = v^\mathbf{b}$. \qed
\end{corollary}
We give another easier equivalent condition, and sketch the proof of the equivalence.

\begin{corollary}
Let $w \in S_k$ be a permutation  and let  $F$ be the set of inversions in $w$.
For each inversion $f= (i, j) \in F$ define the linear polynomial $L(f) = \sum_{k=i}^{j-1}v_k$ and the polynomial $L_w = \prod_{f\in F}L(f) $. Assume that $l(w)=|\mathbf{b}|$.
\begin{enumerate}
  \item With the notations above $\coef(R_w, v^\mathbf{b}) \neq 0$ if and only if $\coef(L_w, v^\mathbf{b}) \neq 0$.
  \item For $s\leq t$ let $K_{s,t}$ denote the number of inversions $(i, j)$ of $w$ such that $s\leq i<j\leq t+1$. Then $\coef(R_w, v^\mathbf{b}) \neq 0$ if and only if
      \begin{equation}\label{eq:bruhat-hall-ineqs}
        \sum_{i=s}^tb_i\geq K_{s,t}
      \end{equation}
      for all $1\leq s\leq t\leq k-1$.
\end{enumerate}
\end{corollary}

\begin{proof}
The second statement is just the usual Hall-theoretical reformulation of the first statement, so it is enough to prove the first.
The direction that $\coef(R_w, v^\mathbf{b}) \neq 0$ implies $\coef(L_w, v^\mathbf{b}) \neq 0$ can be proved by induction on $l(w)$.

For the other direction it is enough to show that there exists a Bruhat chain $ c = \id < u_1 < \cdots < u_{l(w)} = w$, such that $A(c) = L_w$, which can also be proved by induction on $l(w)$. We can find a permutation $w'$ and a transposition $(i, j)$ with $w' \cdot (i, j) = w$, $w(i)= w(j) + 1$, and $l(w)=l(w')+1$. With this choice $F$ consists of the set of inversions in $w'$ and $(i, j)$.
If we apply the induction hypothesis for $w'$ we get a Bruhat chain $c'$ with $A(c') = L(w')$, implying that $A(c) = L_w$ holds for $c = c' \cup (i, j)$.

\end{proof}
Notice that specializing \eqref{eq:bruhat-hall-ineqs} to the longest permutation we recover the system of inequalities \eqref{perfect-system}. We conjecture that the condition \eqref{eq:bruhat-hall-ineqs} is sharp: the grasshopper can survive with permutations under $w$ and forbidden sets $M_i$ with size $b_i$ if and only if \eqref{eq:bruhat-hall-ineqs} holds.

\section{The Symplectic case}\label{sec:sp}
There are natural generalizations of the previous constructions. (We thank Richard Rimanyi for drawing our attention to this possibility.) So far we studied Grassmannians and flag manifolds which are homogeneous spaces for the simple groups $\SL(n)$. There are three other infinite series of complex Lie groups (more exactly simple Lie algebras). Since the corresponding Weyl groups are similar they lead to very similar variants of our previous constructions. It turns out that the case of the symplectic groups is the most convenient. We discuss these statements in somewhat less details. For example studying the degree of the symplectic (or sometimes called Lagrangian) Grassmannians leads to a \emph{signed} Erd\H os--Heilbronn theorem.
\begin{theorem}\label{thm:signed-eh}
If $A$ is a set of distinct non zero residue classes $a_1,\dots, a_{n}$ modulo $p$, $a_i + a_j \neq 0$  and $ p > 2k(n-k) + \binom{k+1}{2}$, then:
\begin{equation*}
 \left| \left\{\sum_{i\in I}\pm a_i\ \big| \ I \subset (1, \dots, n) ,|I| = k\right\}\right | > 2k(n-k) + \binom{k+1}{2}
\end{equation*}
\end{theorem}
This can be proved analogously to our geometric proof for the Dias da Silva--Hamidoune theorem, by replacing the Grassmannian with the symplectic Grassmannian. Notice that the right hand side $d=2k(n-k) + \binom{k+1}{2}$ is exactly the dimension of the symplectic Grassmannian $\SGr_k(\C^{2n})$.

The estimate, unlike the unsigned version, is not sharp. Based on computer calculations we conjecture that the extremal cases for $A$ are the arithmetic progressions
\[a,3a,5a,\dots,(2n-1)a,\]
for which the number of signed sums is $k(2n-k)+\delta(k)$ where $\delta(k)$ is 1 if $k\neq 2$ and 0 if $k=2$. For $k>2$ and $n$ small this can be proved by computer. For $k=2$ the estimate is sharp.

It turned out that there is a simpler proof which is a slight modification of the proof of the Dias da Silva--Hamidoune theorem.
\begin{proof}

Let's have a Schubert variety $\sigma_\lambda\subset \Gr_k(\C^{2n})$, by theorem\ref{deg-of-schubert} we have:
\[\deg(\sigma_\lambda)=\frac{d!}{(I_1-1)! \cdots (I_k-1)!} \cdot \prod_{i<j}(I_j -I_i)=\int\limits_{\Gr_k(\C^{2n})}c_1^d[\sigma_\lambda],  \]
where  $I_j:=2n-k+j-(k-j)=2(n-k+j)$ and $d=\dim(\sigma_\lambda)$. The cohomology class $[\sigma_\lambda]\in H^*\big( \Gr_k(\C^{2n})\big)$ is given by the Schur polynomial $s_\lambda(\mathbf{\alpha})=\det(\alpha_j^{\lambda_i+k-i})/V(\alpha)$ where the $\alpha_i$'s are the Chern roots of $S^*$, the dual of the tautological subbundle. As $s_\lambda$ is symmetric it can be also expressed as a polynomial of the elementary symmetric polynomials of the $\alpha_i$'s, i.e. the Chern classes $ c_i(S^*)$. To apply the ABBV formula we need an equivariant lift of the integrand: we replace the Chern classes $ c_i(S^*)$ with the torus-equivariant ones using that $S^*$ is a $T(2n)$-bundle. Using the notation $\widehat{[\sigma_\lambda]}$ for the lift of $[\sigma_\lambda]$ we have that
\[\int\limits_{\Gr_k(\C^{2n})}c_1^d[\sigma_\lambda]=\int\limits_{\Gr_k(\C^{2n})}\hat c_1^d\widehat{[\sigma_\lambda]}.  \]

 For $\lambda=(k-1, \dots, 0)$ we have
\[s_\lambda(\alpha)=\prod\limits_{1\leq i < j \leq k}(\alpha_i+\alpha_j).\]

Applying now the ABBV formula and using the usual dehomogenizing trick with a $d$-element set $M$ we get
\[\frac{d!}{(I_1-1)! \cdots (I_k-1)!} \cdot \prod_{i<j}(I_j -I_i) = \sum\limits_{J\in \binom{2n}{k}}\frac{\prod\limits_{m\in M}(x_J - m) \cdot \prod\limits_{i,j\in J, i < j }(x_i+x_j)}{\prod\limits_{i\in J}  \   \prod\limits_{j\notin J} (x_i-x_j)}.\]

%the identity (\ref{eq:azonossag2}) on the degree of Schubert varieties
%\[\frac{d!}{(I_1-1)! \cdots (I_k-1)!} \cdot \prod_{i<j}(I_j -I_i) = \sum\limits_{J\in \binom{2n}{k}}\frac{\prod\limits_{m\in M}(x_J - m) \cdot [\sigma_I]_T|_J }{\prod\limits_{i\in J}  \   \prod\limits_{j\notin J} (x_i-x_j)}\]
%to the Schubert variety $\sigma_\lambda\subset \Gr_k(\C^{2n})$ for $\lambda=(k-1, \dots, 0)$ and $|M|=d$. The cohomology class $[\sigma_\lambda]\in H^*\big( \Gr_k(\C^{2n})\big)$ is given by the Schur polynomial $s_\lambda(\mathbf{x})$ where the $x_i$'s are Chern roots of the dual of the tautological subbundle. For $\lambda=(k-1, \dots, 0)$ we have
%\[s_\lambda(\mathbf{x})=\prod\limits_{1\leq i < j \leq k}(x_i+x_j).\]
%Then, using also the conversion formula (\ref{conversion}) we arrive to the identity
%\[\frac{d!}{(I_1-1)! \cdots (I_k-1)!} \cdot \prod_{i<j}(I_j -I_i) = \sum\limits_{J\in \binom{2n}{k}}\frac{\prod\limits_{m\in M}(x_J - m) \cdot \prod\limits_{1\leq i < j \leq k}(x_i+x_j)}{\prod\limits_{i\in J}  \   \prod\limits_{j\notin J} (x_i-x_j)},\]
%where $I_j:=2n-k+j-(k-j)=2(n-k+j)$.

Then substitute the elements of $A \cup -A$ into the $x_i$'s. Notice that because of  the factor $\prod\limits_{i,j\in J, i < j }(x_i+x_j)$ the only non zero terms correspond to subsets not containing both $a$ and $-a$. Then we can use the argument of Theorem \ref{thm:eh}. by choosing $M$ as the possible $k$-term signed sums to finish the proof.
\end{proof}
\begin{remark} One can prove the signed Erd\H os--Heilbronn theorem using the Combinatorial Nullstellensatz to the polynomial
\[ P(\mathbf{x}) = \prod\limits_{1\leq i < j \leq k}(x_{j}^{2}  - x_{i}^{2})\prod\limits_{b \in B}\left(\sum\limits_{i=1}^{ k}x_i -b  \right) .\]
where $B\subset \F_p$ and $|B|=d$. The coefficient of $\prod_{1\leq i \leq k}x_{i}^{2n-i} $ can be calculated using the standard theory of symmetric polynomials, namely with the hook rule. This coefficient is exactly the degree of the Schubert variety $\sigma_\lambda$, and the calculation is a proof for this special case of the Schubert degree formula.

\subsection{The small $p$ case} For $p\leq d=2k(n-k)+\binom{k+1}2$ we want to find a Schubert variety $\sigma_\lambda\subset \Gr_k(\C^{2n})$ with dimension $p-1$ and we also want its cohomology class to be divisible by $\prod\limits_{1\leq i < j \leq k}(\alpha_i+\alpha_j)$. This can be achieved by substituting $\alpha_i^2$'s into a Schur polynomial $s_{\lambda}$. Then
\[s_{\lambda}(\alpha_1^2,\dots,\alpha_k^2)=\frac{\det(\alpha_j^{2\lambda_i+2k-2i})}
{\prod\limits_{1\leq i < j \leq k}(\alpha_i+\alpha_j)(\alpha_i-\alpha_j)}, \]
implying that
\[s_\mu(\alpha)=s_{\lambda}(\alpha_1^2,\dots,\alpha_k^2)\prod\limits_{1\leq i < j \leq k}(\alpha_i+\alpha_j), \]
for $\mu_i=2\lambda_i+k-i$. The only problem is that the degree of $s_{\lambda}(\alpha_1^2,\dots,\alpha_k^2)$ is even, so we can find a desired Schubert variety with dimension $p-1$ only if $k$ is even (using that $k$ is congruent to $\binom k2$+$\binom {k+1}2$ modulo 2). Therefore we arrived at the following:
\begin{theorem}\label{thm:signed-eh-smallp}
If $A$ is a set of distinct non zero residue classes $a_1,\dots, a_{n}$ modulo $p$, $a_i + a_j \neq 0$  and $ p \leq 2k(n-k) + \binom{k+1}{2}$, then:
\begin{equation*}
 \left| \left\{\sum_{i\in I}\pm a_i\ \big| \ I \subset (1, \dots, n) ,|I| = k\right\}\right | =p,
\end{equation*}
if $k$ is even. If $k$ is odd, then the zero residue class might be missing.
\end{theorem}
We conjecture however that even if $k$ is odd, the zero residue class is not missing.

\end{remark}
\bigskip

The grasshopper problem also has a symplectic version, we discuss it in the following:

\subsection{Grasshopper with signs}
Let $a_1,a_2,\dots,a_{k}$ be distinct non negative integers and let $|M_i| = b_i$ sets for $1\leq i \leq k$.  A grasshopper is to jump along the real axis, starting at the point 0 and make ${k}$ jumps with lengths $a_1,a_2,\dots,a_{k}$ in some order, and at each step the grasshopper can decide if he decides to jump to the left or to the right. How many forbidden positions can be given at each step, such that the grasshopper can avoid them?

Or more formally:
what are the integer sequences $(b_1, \dots, b_k) $, for which for all non negative $a_1,a_2,\dots,a_{k}$ sequences there is a permutation $\pi \in S_k$ and a sign function $ s : (1, \dots, k) \to \{-1, +1\}$, such that
   \[  \sum\limits_{1 \leq j \leq i} s(j)a_{\pi(j)}  \notin M_i?  \]

%For a bit suprise the answer of this question lies in the heart of the equivariant topology of the symplectic flag manifold $S\fl$ with the %action of $ Sp(2k, \C)$.

Let $\Gamma_{\lambda}$ be the irreducible representation of $Sp(2k, \C)$ with highest weight $\lambda=(\lambda_1>\lambda_2>\cdots>\lambda_{k-1}>\lambda_{k} > 0)$.
The minimal orbit in $\P(\Gamma_\lambda)$ is the  symplectic flag variety $S\fl= S\fl(k)$. Its degree in $\P(\Gamma_\lambda)$ is calculated by the integral
  \[  \deg \big(S\fl \subset \P(\Gamma_\lambda)\big)=\int\limits_{\fl}\hat c_1(L_\lambda)^d, \]
where $L_\lambda$ is the canonical bundle---the pullback of the canonical bundle over $\P(\Gamma_\lambda)$,  $\hat c_1(L_\lambda)$ is the first $T$-equivariant Chern class of  $L_\lambda$  and $d=\dim (S\fl)= k^2$.

By the Borel--Hirzebruch formula we have:
 \[  \deg \big(S\fl \subset \P(\Gamma_\lambda)\big) =
 \frac{d!}{k!} \prod_{i=1}^{k}\lambda_i\prod\limits_{i<j \leq k}\frac{\lambda_i^2- \lambda_j^2}{j^2- i^2}.  \]

On the other hand the fixed points of the $T(k)$-action---where $T(k)$ is a maximal torus of $Sp(2k, \C)$---are parametrized by the pairs $ (\pi,  s)$ where  $\pi \in S_k$ is a permutation  and $ s : (1, \dots, k) \to \{-1, +1\}$ is a sign function. Applying  the ABBV integral formula we get:

\begin{equation} \label{eq:deg-sfl}
\deg \big(S\fl \subset \P(\Gamma_\lambda)\big)=\int\limits_{\fl}c_1^T(L)^d = \sum\limits_{(\pi,s)}\frac{\left(\sum\limits_{i=1}^{k}s(i)\lambda_i x_{\pi(i)}\right)^d}{e(T_{(\pi,s)}S\fl)},
\end{equation}

where
  \[ e(T_{(\pi,s)}S\fl)=\prod\limits_{1\leq i \leq j\leq k} s(i)x_{\pi(i)}+s(j)x_{\pi(j)} \prod\limits_{1\leq i < j\leq k} s(i)x_{\pi(i)}-s(j)x_{\pi(j)}.\]

The right hand side of (\ref{eq:deg-sfl}) is a polynomial in the variables $\lambda_i$.

We define the homogenous polynomial $P$ in the variables $v_i$ with degree $d$ by substituting $ \lambda_i = \sum\limits_{j=i}^{k}v_{j}$.
The coefficient of $v^\mathbf{b}$ for $\mathbf{b}=(b_1,\dots,b_{k})$ is:

\begin{equation}
\coef(P,v^\mathbf{b})= \binom{d}{\mathbf{b}} \sum\limits_{(\pi,s) } \frac{ \prod\limits_{i=1}^{k}\left(\sum\limits_{j=1}^{i} s(j)x_{\pi(j)}\right)^{b_i}}{e(T_{(\pi,s)}S\fl)}.
\end{equation}

By the usual trick we can add lower order terms:

\begin{equation}
\coef(P,v^\mathbf{b})= \sum\limits_{(\pi,  s ) }\frac{\prod\limits_{i=1}^{k}  \  \prod\limits_{m \in M_i }\left(\sum\limits_{j=1}^{i}  s(j)x_{\pi(j)}  -m\right)}{e(T_{(\pi,s)}S\fl)}.
\end{equation}

We can substitute the numbers $a_i$ into the variables $x_i$ because $a_i \neq a_j$ if $i \neq j$ and $a_i \neq -a_j$, so the denominator is not zero.
This implies that the grasshopper can always jump if $\sum\limits_{1\leq i \leq k}b_i = k^2$ and $\coef(P,v^\mathbf{b}) \neq 0$.

The Borel--Hirzebruch formula gives:
%d! \prod\limits_{1 \leq i\leq j \leq k}\frac{\sum\limits_{i\leq l \leq k}v_{l}+ \sum\limits_{j\leq l \leq k}v_l}{2k+2-i-j} \prod\limits_{i<j \leq k}\frac{ \sum\limits_{i\leq l < j}v_{l}}{j-i}
\begin{equation}
P(\mathbf{v}) = K\cdot \prod\limits_{1 \leq i\leq j \leq k}\left(\sum\limits_{l=i}^{k}v_{l}+ \sum\limits_{l=j}^{k}v_{l}\right)
\prod\limits_{1\leq i<j \leq k}\left(\sum\limits_{l=i}^{j-1}v_{l}\right),
\end{equation}
where the $K$ is a positive constant.
We can see that $\coef(P,v^\mathbf{b}) \neq 0$ if and only if $\coef(Q,v^\mathbf{b}) \neq 0$, where:

\begin{equation}
Q(\mathbf{v}) = \prod_{1\leq i\leq j \leq k}\left(\sum_{l=i}^{k}v_l \right) \prod_{ 1 \leq i<j \leq k} \left( \sum_{l=i}^{j-1}v_l \right)
\end{equation}

We can assign a bipartite graph to $\mathbf{b}$ similarly to Definition \ref{de:bipartite}. such that the non vanishing of $\coef(Q,v^\mathbf{b})$ is equivalent to the existence of a perfect matching. Slightly modifying the argument of Proposition \ref{hall}. we arrive at the following:

\begin{theorem}

Let $a_1,a_2,\dots,a_{k}$ be distinct non negative integers,  and let $M_i$ for $1\leq i \leq k$ be sets of integers with $\sum\limits_{1\leq i \leq k}b_k= k^2$ for $|M_i| = b_i$. Assume moreover that the following conditions hold:
\begin{enumerate}
  \item If $1 \leq i \leq j \leq k-1$ then $\sum\limits_{i\leq l\leq j}b_l \geq \binom{j-i+2}{2}$,
  \item $1 \leq i  \leq k$, then $\sum\limits_{i\leq l\leq k}b_l \geq (k-i+1)^2$.
\end{enumerate}

Then there is a permutation $\pi \in S_k$ and a sign function $ s : (1, \dots, k) \to \{-1, +1\}$, such that for all $1\leq i \leq k$
   \[  \sum\limits_{1 \leq j \leq i} s(j)a_{\pi(j)}  \notin M_i.  \]
\end{theorem}
\begin{remark} It is not difficult to see that the signed Erd\H os--Heilbronn theorem over the integers follows from the signed grasshopper theorem, exactly as for the unsigned versions.
\end{remark}

\section{Concluding Remarks}
It turns out that all results in the paper can be reproved using the Combinatorial Nullstellensatz. Nevertheless we hope that the advantages of our geometric method is clear: the right estimate is suggested by the dimension of a homogeneous space, and the non-vanishing of the relevant coefficient is guaranteed since it is the degree of a variety. This identification also helps to use existing formulas like the Borel--Hirzebruch formula to calculate this coefficient and study its prime factors.

It is tempting to apply this geometric approach to problems where the coefficient formula was succesfully used like the Dyson identity \cite{Karasev2012} or the $q$-Dyson identity \cite{q-dyson}. In fact if we specialize for $n=2$ the proof Karasev and Petrov gave to the Dyson identity in \cite{Karasev2012}, then it is  essentially equivalent to the calculation of the degree of the Segre embedding using the identity \eqref{eq:deg-segre2}.  It also looks promising to approach the inverse Erd\H os--Heilbronn problem this way.

\appendix
\section{The Borel--Hirzebruch formula}\label{app:bh} The main references are \cite{fulton-harris} and \cite{gross-wallach}. Let $G$ be a simple complex Lie group ($\SL(n)$ or $\Sp(2n)$ in the paper). Given a dominant weight $\lambda$ we have an irreducible representation $\Gamma_\lambda$ of $G$ with highest weight $\lambda$ we have an action of $G$ on the projective space $\P(\Gamma_\lambda)$. Its minimal orbit $X_\lambda\subset \P(\Gamma_\lambda)$ is a homogeneous space for $G$. Let $R^+$ denote the positive roots of $G$. For a root $\alpha\in R^+$ let $\alpha^\vee$ denote the corresponding coroot. Unfortunately there are different conventions, we use the one that a coroot is an element of the Cartan subalgebra. This corresponds to the $H_\alpha$ notation of Fulton and Harris in \cite{fulton-harris}. The Borel--Hirzebruch formula calculates the projective degree of the minimal orbit:

\begin{theorem} \label{thm:bh} With the notation above
\begin{equation*}
  \deg\big( X_\lambda\subset \P(\Gamma_\lambda)\big)=
        d!\prod_{\alpha\in T_\lambda}  \frac{\langle \lambda,\alpha^\vee\rangle}  {\langle \rho,\alpha^\vee\rangle},
\end{equation*}
where $T_\lambda=\{\alpha\in R^+: \langle \lambda,\alpha^\vee\rangle\neq0\}$, $d=|T_\lambda|$ and $\rho$ is half the sum of the positive roots.
\end{theorem}
The pairing $\langle \cdot,\cdot\rangle$ is the dual pairing. $d$ is also the dimension of $X_\lambda$.
\subsection{$\GL(n)$}
We switch to $\SL(n)$ since $\GL(n)$ is not a simple group. The Cartan subalgebra consists of the zero trace diagonal matrices. The weights are generated by the functionals $L_i$ for $i=1,\dots,n$ with value the $i$-th entry of the diagonal matrix. Notice that $\sum_{i=1}^nL_i=0$. The positive roots  are
\[R^+=\{L_i-L_j:i<j\}.\]
Half the sum of the positive roots is
  \[\rho=\sum_{i=1}^{n-1}(n-i)L_i.\]
Let $H_i$ be the standard basis for the diagonal matrices, then
  \[(L_i-L_j)^\vee=H_i-H_j.\]
The dominant weights are of the form $\sum_{i=1}^n\lambda_iL_i$ with $\lambda_1\geq\lambda_2\geq\cdots\geq \lambda_n$. Because of the relation $\sum_{i=1}^nL_i=0$ you can choose $\lambda_n$ to be 0. The minimal orbit is a partial flag manifold corresponding to the multiplicities of the pairwise distinct $\lambda_i$'s. In particular the Grassmannian $\gr_k(\C^n)$ corresponds to $\lambda=(1,1,\dots,1,0,\dots,0)$ with $k$ copies of 1's.
\subsection{$\Sp(2n)$}
 The Cartan subalgebra of the symplectic group $\Sp(2n)$ can be identified with the space of $2n\times 2n$ diagonal matrices with $a_{n+i,n+i}=-a_{i,i}$. The weights are generated by the functionals $L_i$ for $i=1,\dots,n$ with value $a_{i,i}$.  The positive roots  are
  \[  R^+=\{L_i-L_j:i<j\}\cup \{L_i+L_j:i\leq j\}.    \]
Half the sum of the positive roots is
  \[\rho=\sum_{i=1}^{n}(n-i+1)L_i.\]
Let $H_i$ be the diagonal matrix with $a_{i,i}=-a_{n+i,n+i}=1$, and the other entries 0. Then
  \[(L_i-L_j)^\vee=H_i-H_j, \ \ (L_i+L_j)^\vee=H_i+H_j \ \text{for }\  i<j, \ \ \text{and} \ (2L_i)^\vee=H_i.\]
The dominant weights are of the form $\sum_{i=1}^n\lambda_iL_i$ with $\lambda_1\geq\lambda_2\geq\cdots\geq \lambda_n$. The minimal orbit is the symplectic flag manifold for pairwise distinct $\lambda_i$'s. The symplectic Grassmannian $S\gr_k(\C^{2n})$ corresponds to $\lambda=(1,1,\dots,1,0,\dots,0)$ with $k$ copies of 1's.

\section{The ABBV formula for minimal orbits} \label{app:abbv}
A maximal torus $T(n)$ of $G$ is acting on the minimal orbit $X_\lambda$. The fixed points can be identified with the orbit of $\lambda$ under the action of the Weyl group $W_G$ of $G$, i.e with the cosets $W_G/W_\lambda$, where $W_\lambda$ is the stabilizer subgroup. The tangent space in the fixed point corresponding to $\lambda$ has weights
$\tau_\lambda:=\{\alpha\in R^-: \langle \lambda,\alpha^\vee\rangle\neq0\}$. To get the weights for the other fixed points you apply the Weyl group. Therefore the integral of a $T(n)$-equivariant cohomology class $\alpha$ is
  \[\int\limits_{X_\lambda}\alpha=\sum_{f\in W_G/W_\lambda}  \frac{\alpha|_f}{\prod\limits_{\omega\in \tau_\lambda} f(\omega)}.\]
To calculate the degree we need the restrictions of the first $T(n)$-equivariant Chern class of the canonical bundle:
\[\hat c_1(L_\lambda)|_f=\sum_{i=1}^n\lambda_if(x_i),\]
where $H^*_{T(n)}=\Z[x_1,\dots,x_n]$ and the $x_i$'s are the \idez{positive} generators. So there is a double twist in the notation. For $\GL(n)$ and $\Sp(2n)$ the $x_i$'s correspond to the $-L_i$'s, and the $L_i$'s correspond to the weights of the tautological subbundles (as opposed to the canonical subundles).
\subsection{$\GL(n)$} The Weyl group is the symmetric group $S_n$ acting on the weights by permuting the coefficients. For $\lambda=(1,0,\dots,0)$ the minimal orbit is the  whole projective space $\P(\C^n)$ and, more generally for $\lambda=(1,\dots,1,0,\dots,0)$ with $k$ copies of 1's the minimal orbit is the Grassmannian $\gr_k(\C^n)$ in the Pl\"ucker embedding.
\subsection{$\Sp(2n)$} The Weyl group is the semidirect product of the symmetric group $S_n$ and $\Z_2^n$ acting on the weights by permuting the coefficients and multiplying the $i$-th coefficient by $-1$, respectively.
\bibliography{eh}
\bibliographystyle{alpha}
\end{document}